\documentclass{amsart}
\usepackage[utf8]{inputenc}
\usepackage[english]{babel}
\usepackage{float}
\usepackage{nicefrac}
\usepackage{amsmath}
\usepackage{amssymb}
\usepackage{amsthm}
\usepackage{tikz-cd}
\usepackage{tikz}
\usepackage{caption}
\usepackage{braket}

\theoremstyle{definition}
\newtheorem{thm}{Theorem}[section]

\newtheorem{prop}[thm]{Proposition}
\newtheorem{defn}[thm]{Definition}

\numberwithin{thm}{subsection}

\makeatletter
\newcommand{\proofpart}[2]{%
	\par
	\addvspace{\medskipamount}%
	\noindent\emph{$\bullet$ Part #1: #2}\par\nobreak
	\addvspace{\smallskipamount}%
	\@afterheading
}
\makeatother

\renewcommand{\epsilon}{\varepsilon}

\renewcommand{\phi}{\varphi}

\renewcommand{\subset}{\subseteq}

\newcommand{\A}{{\mathbb A}}

\newcommand{\Q}{{\mathbb Q}}
\newcommand{\R}{{\mathbb R}}
\renewcommand{\L}{{\mathbb L}}
\newcommand{\C}{{\mathbb C}}

\title{The dlt motivic zeta function is not well-defined}
\author{Johannes Nicaise}
\address{Imperial College,
Department of Mathematics, South Kensington Campus,
London SW72AZ, UK, and KU Leuven, Department of Mathematics, Celestijnenlaan 200B, 3001 Heverlee, Belgium.} \email{j.nicaise@imperial.ac.uk}

\author{Naud Potemans}
\address{Department of Mathematics, Celestijnenlaan 200B, 3001
			Heverlee, Belgium}
\email{naud.potemans@kuleuven.be}

\author{Willem Veys}
\address{Department of Mathematics, Celestijnenlaan 200B, 3001
			Heverlee, Belgium}
\email{wim.veys@kuleuven.be}
\thanks{The authors are partially supported by grants G079218N and G0B17121N
	of the Fund for Scientific Research -- Flanders (FWO), long term structural funding (Methusalem grant) of the Flemish Government, and EPSRC grant EP/S025839/1.}

\begin{document}

\begin{abstract}
 In \cite{Xu1}, Xu defines the dlt motivic zeta function associated to a regular function $f$ on a smooth variety $X$ over a field of characteristic zero. This is an adaptation of the classical motivic zeta function that was introduced by Denef and Loeser \cite{DL2}. The dlt motivic zeta function is defined on a dlt modification via a Denef-Loeser-type formula, replacing classes of strata in the Grothendieck ring of varieties by stringy motives. We provide explicit examples that show that the dlt motivic zeta function depends on the choice of dlt modification, contrary to what is claimed in \cite{Xu1}, and that it is therefore not well-defined. 	
\end{abstract}	
	\maketitle
\section{Introduction}
Let $k$ be a field of characteristic $0$. Let $X$ be a smooth $k$-variety and $f$ a non-constant regular function on $X$. In \cite{DL2}, Denef and Loeser define the motivic zeta function of $f$, a formal power series over a suitable Grothendieck ring of $k$-varieties that should be viewed as a motivic upgrade of Igusa's local zeta function associated with a polynomial over a $p$-adic field. The motivic zeta function captures the geometric structure underlying Igusa's local zeta function and is a very rich invariant of the singularities of the hypersurface $Z(f)$ defined by $f=0$. It is defined intrinsically in terms of a motivic integral on $X$, but it can be calculated explicitly on a log resolution of $(X,Z(f))$. In this way, it allows one to extract invariants from log resolutions that are independent of the choice of the log resolution.

 The most important question around the motivic zeta function is the {\em monodromy conjecture}, which expresses a precise relation between the poles of the motivic zeta function and the local monodromy eigenvalues of $f$. A key problem in this context is that most of the apparent poles in the explicit formula on a log resolution tend to cancel out in practice, for reasons which are not yet understood. To analyze this behaviour, it is natural to try to compute motivic zeta functions on suitable partial resolutions of $(X,Z(f))$, to reduce the set of candidate poles.
	
In \cite{Xu1}, Xu proposed a variant of the motivic zeta function, called the dlt motivic zeta function. At first sight, it no longer has an intrinsic definition in terms of $X$ and $f$; rather, Xu proposed an explicit formula in terms of a dlt modification of $(X,Z(f)_{red})$, which should be viewed as a minimal partial resolution of singularities of the pair. This formula involves the so-called stringy motives of the strata of the total transform of $Z(f)$. The central claim of \cite{Xu1} is then that the dlt motivic zeta function is independent of the choice of the dlt modification. Unfortunately, this claim is incorrect; the goal of the present article is to provide counter-examples and to locate the errors in the arguments in \cite{Xu1}.	

\section{Preliminaries and notation}
\subsection{Singularities of pairs}
	Throughout this article, we work over a base field $k$ of characteristic 0. A $k$-variety
will be an integral separated $k$-scheme of finite type.	
	
	 We will largely follow the definitions of \cite{KoMMP}, but we use a more restrictive definition of a {\em pair}: in this article, a pair $(X,\Delta)$ consists of a normal quasi-projective $k$-variety $X$ together with an effective $\Q$-divisor $\Delta$ such that $K_X+\Delta$ is $\Q$-Cartier. We say that $\Delta$ is a boundary if its coefficients all lie in $[0,1]$.
	
	Let $(X,\Delta)$ be a pair. Consider a birational morphism $\pi:Y\rightarrow X$ where $Y$ is a normal $k$-variety. Write
	\[
	K_Y-\pi^*(K_X+\Delta)=\sum_E d(E,X,\Delta)E,
	\]
	where $E\subset Y$ are distinct prime divisors. Note that the left hand side of this expression is well-defined as a $\Q$-divisor, and not merely up to $\Q$-linear equivalence.
		We define the \textit{log discrepancy} $a(E,X,\Delta)$ of $E$ associated with the pair $(X,\Delta)$ as $a(E,X,\Delta)=d(E,X,\Delta)+1$.
	Moreover, we define the discrepancy of $(X,\Delta)$ by
	\[
	\mathrm{discrep}(X,\Delta)=\inf_E d(E,X,\Delta)
	\] where $E$ runs over all exceptional prime divisors of all birational morphisms $\pi\colon Y\to X$ with $Y$ a normal $k$-variety.
	
 We say that $(X,\Delta)$ is {\em snc} if $X$ is smooth and $\Delta$ is a reduced divisor with strict normal crossings.	We call $(X,\Delta)$ \textit{log canonical} (lc) if $\mathrm{discrep}(X,\Delta)\geq-1$, and  \textit{Kawamata log terminal} (klt) if $\mathrm{discrep}(X,\Delta)>-1$ and all coefficients of $\Delta$ lie in $[0,1)$. 	If $(X,\Delta)$ is a log canonical pair, then a closed subset $Z\subset X$ is called a {\em log canonical center} of $(X,\Delta)$ if there exists a prime divisor $E$ over $X$ such that $a(E,X,\Delta)=0$ and the center of $E$ in $X$ is $Z$.

	We say that a log canonical pair $(X,\Delta)$ is \textit{divisorial log terminal} (dlt) if  there is an open set $U\subseteq X$ such that $(U,\Delta\vert_U)$ is snc and $U$ meets every log canonical center of $(X,\Delta)$. If $(X,\Delta)$ is dlt, then the log canonical centers of $(X,\Delta)$ are precisely the connected components of intersections of $r$ distinct components of multiplicity $1$ in $\Delta$, for all $r\geq 1$, by Section 3.9 of \cite{Fuj}. Each of these log canonical centers is normal, and its codimension is equal to $r$.

	\begin{defn}
		Consider a pair $(X,\Delta)$ with $\Delta$ a boundary. A {\em dlt modification} of $(X,\Delta)$ is a proper birational morphism $g:X^{dlt}\rightarrow X$ such that
		\begin{enumerate}
			\item[i.] $X^{dlt}$ is $\Q$-factorial,
			\item[ii.] $(X^{dlt},\Delta^{dlt})$ is dlt, where $\Delta^{dlt}\subset X^{dlt}$ is the sum of the birational transform of $\Delta$ and the reduced exceptional divisor of $g$,
			\item[iii.] $K_{X^{dlt}}+\Delta^{dlt}$ is $g$-nef, that is, $(K_{X^{dlt}}+\Delta^{dlt})\cdot C\geq0$ for all irreducible curves $C\subset X^{dlt}$ that are contracted to a point by $g$.
		\end{enumerate}
 To make the divisor $\Delta$ explicit, we	will also say that $g:(X^{dlt},\Delta^{dlt})\rightarrow (X,\Delta)$ is a dlt modification.
	\end{defn}
	
	Such a dlt modification always exists, by Theorem 3.1 in \cite{KoKo}. It is not unique in general, but any two dlt modifications of $(X,\Delta)$ are crepant birational.
	
\subsection{The different and the stringy motive}
	Let $(X,\Delta)$ be a dlt pair and let $Z$ be a log canonical center of $(X,\Delta)$. We denote by $U$ the maximal open subscheme of $X$ such that $(U,\Delta\vert_U)$ is snc and such that $U$ only intersects those irreducible components of $\Delta$ that contain $Z$. Then there exists a unique $\Q$-divisor $\mathrm{Diff}_Z(\Delta)$ on $Z$, called the {\em different}, such that for all sufficiently divisible even integers $m>0$, the Poincar\'e residue isomorphism
$$\omega^{\otimes m}_{U}(m\Delta\vert_U)\vert_{Z\cap U}\to \omega^{\otimes m}_{Z\cap U}$$ extends to an isomorphism
	$$\omega^{[m]}_{X}(m\Delta)\vert_{Z}\to \omega^{[m]}_{Z}(m\mathrm{Diff}_{Z}(\Delta)).$$
		 See Definition 4.2 in \cite{KoMMP}, where the different was denoted by $\mathrm{Diff}^*$.
		 We recall that $\omega^{[m]}_{X}(m\Delta)$ denotes the reflexive hull of the sheaf $\omega^{\otimes m}_X(m\Delta)$ of $m$-canonical forms with poles bounded by $m\Delta$, and similarly for $\omega^{[m]}_{Z}(m\mathrm{Diff}_{Z}(\Delta))$. If $m$ is sufficiently divisible then $m(K_X+\Delta)$ is Cartier and $\omega^{[m]}_{X}(m\Delta)$ is the corresponding line bundle.
		
		  The definition implies in particular that 	
		\[
		(K_X+\Delta)\vert_Z \sim_{\Q} K_Z+\mathrm{Diff}_Z(\Delta)
		\]
but $\mathrm{Diff}_Z(\Delta)$ is well-defined as a $\Q$-divisor, and not merely up to $\Q$-linear equivalence.	  By Theorem 4.19 in \cite{KoMMP}, the different $\mathrm{Diff}_Z(\Delta)$ is effective, and the pair $(Z,\mathrm{Diff}_Z(\Delta))$ is again dlt.

When $(X,\Delta)$ is snc, then $\mathrm{Diff}_Z(\Delta)=\Delta'\vert_Z$ where $\Delta'$ is the sum of the prime components of $\Delta$ that do not contain $Z$; this is the usual adjunction formula for snc pairs. This result can be generalized in the following way, leading to a practical method to compute differents on log resolutions.
		
	\begin{prop} Let $(X,\Delta)$ be a dlt pair, and let $Z$ be a log canonical center of $(X,\Delta)$. Let $\pi\colon Y\rightarrow X$ be a log resolution of $(X,\Delta)$ such that $\pi$ is an isomorphism over the generic point of $Z$. Let $\Delta_1,\ldots,\Delta_r$ be the irreducible components of $\Delta$ that contain $Z$. For each $i\in \{1,\ldots,r\}$ we denote by $\tilde{\Delta}_i$ the birational transform of $\Delta_i$ on $Y$. Similarly, we write $\tilde{Z}$ for the birational transform of $Z$ on $Y$.
		We set	
		\[\Delta_Y=\pi^*(K_X+\Delta)-(K_Y+\sum^r_{i=1}\tilde{\Delta}_i).\]
		 Then $\tilde{Z}$ is not contained in the support of $\Delta_Y$, so that $\Delta_Y\vert_{\tilde{Z}}$ is a well-defined $\Q$-divisor on $\tilde{Z}$. We can express the different $\mathrm{Diff}_Z(\Delta)$ as
		\[
		\mathrm{Diff}_Z(\Delta)=(\pi\vert_{\tilde{Z}})_*(\Delta_Y\vert_{\tilde{Z}}).
		\]
\end{prop}
\begin{proof}
By the definition of a dlt pair, each of the prime divisors $\Delta_i$ appears with coefficient $1$ in $\Delta$, so that $\tilde{Z}$ is not contained in the support of $\Delta_Y$.
 For notational convenience, we set $D=\Delta_Y\vert_{\tilde{Z}}$.

 For every even integer $m>0$, we can consider the Poincar\'e residue isomorphism
$$\omega_Y^{\otimes m}(m\sum_{i=1}^r \tilde{\Delta}_i)\vert_{\tilde{Z}}\to  \omega^{\otimes m}_{\tilde{Z}}.$$
When $m$ is sufficiently divisible, then, twisting by the line bundle $\mathcal{O}_{\tilde{Z}}(mD)$, we obtain an isomorphism
$$(\pi\vert_{\tilde{Z}})^*(\omega^{[m]}_X(m\Delta)\vert_{Z})=\omega_Y^{\otimes m}(m\Delta_Y+m\sum^r_{i=1} \tilde{\Delta}_i)\vert_{\tilde{Z}}\to  \omega^{\otimes m}_{\tilde{Z}}(mD)$$
and taking the pushforward with respect to $\pi\vert_{\tilde{Z}}$ we get an isomorphism
$$\omega^{[m]}_X(m\Delta)\vert_{Z} \to (\pi\vert_{\tilde{Z}})_*\omega^{\otimes m}_{\tilde{Z}}(mD)$$ that coincides with the Poincar\'e residue isomorphism at the generic point of $Z$.

 Let $V$ be an open neighbourhood of the generic point of $Z$ in $X$ such that the pair $(V,\Delta|_V)$ is snc, we have $\Delta=\sum_{i=1}^r\Delta_i$ on $V$, and $\pi$ restricts to an isomorphism  $\pi^{-1}(V)\to V$.
  Then it suffices to show that for all sufficiently divisible integers $m>0$, the sheaves
$(\pi\vert_{\tilde{Z}})_*\omega^{\otimes m}_{\tilde{Z}}(mD)$ and  $\omega_Z^{[m]}((\pi\vert_{\tilde{Z}})_*mD)$
 are equal as subsheaves of the pushforward of $\omega^{\otimes m}_{V\cap Z}$ to $Z$.
  The first of these two sheaves is a line bundle because it is isomorphic to $\omega^{[m]}_X(m\Delta)\vert_{Z}$, and the second is reflexive by definition, so it suffices to check the equality at codimension $1$ points of $Z$. Over such points, $\pi\vert_{\tilde{Z}}$ is an isomorphism because $Z$ is normal, so that the equality is trivially satisfied.
\end{proof}

The final ingredient we need to define the dlt zeta function is the stringy motive of a klt pair.
	\begin{defn}[Definition 3.7 of \cite{Ba2}, Definition 7.7 of \cite{Ve}] 	\label{def:stringy}
		Consider a klt pair $(X,\Delta)$. Let $Y\rightarrow X$ be a log resolution of $(X,\Delta)$ and let $E_i,i\in S$, be the prime divisors appearing in the exceptional locus or in the birational transform of $\Delta$ on $Y$.
 For every subset $I$ of $S$, we define $E_I=\cap_{i\in I}E_i$ and $\mathring{E}_I=E_I\setminus \cup_{j\notin I}E_j$.
		The {stringy motive} $\mathcal{E}_{st}(X,\Delta)$ of $(X,\Delta)$ is defined as
		\[
		\sum\limits_{I\subset S}[\mathring{E}_I]\prod\limits_{i\in I}\frac{\mathbb{L}-1}{\mathbb{L}^{a_i}-1},
		\]
		where $a_i$ is the log discrepancy $a(E_i,X,\Delta)$. Remark that $a_i>0$ for all $i\in S$, since $(X,\Delta)$ is klt. The stringy motive $\mathcal{E}_{st}(X,\Delta)$ is an element of $\tilde{\mathcal{M}}_k^e$, a suitable finite extension of the completed Grothendieck ring of varieties over $k$. Moreover, the stringy motive does not depend on the choice of log resolution.
	\end{defn}
	
	The precise definition of $\tilde{\mathcal{M}}_k^e$ will not be important for our purposes, we refer to \cite{Ba2} and \cite{Ve} for further background.

\subsection{The dlt motivic zeta function}		\label{defXu}
We can now recall Xu's definition of the dlt motivic zeta function (see Section 1 of \cite{Xu1}).
		Consider a non-constant regular function $f$ on a smooth quasi-projective $k$-variety $X$ and let $Z(f)$ be the closed subscheme of $X$ defined by $f=0$. Let
		$$g:(X^{dlt},Z(f)^{dlt}) \rightarrow (X,Z(f)_{red})$$ be a dlt modification. Let $E_i,i\in S$, be the prime divisors appearing in the exceptional locus or in the birational transform of $\Delta$ on $Y$.  Write
		$$g^*(Z(f))=\sum_{i\in S}N_iE_i,\quad K_{Y/X}=\sum_{i\in S}(\nu_i-1)E_i.$$ For every subset $I$ of $S$, we set $E_I=\cap_{i\in I}E_i$ and $\mathring{E}_I=E_I\setminus \cup_{j\notin I}E_j$.
		We denote the different $\mathrm{Diff}_{E_I}(Z(f)^{dlt})$ by $D_I$. This is a slight abuse of notation, since $E_I$ might consist of different connected components. In that case, we simply compute the different on each individual component and then add all of these differents together. We set $\mathring{D}_I=D_I\vert_{\mathring{E}_I}$; then $(\mathring{E}_I,\mathring{D}_I)$ is a klt pair.
		
		Finally, we define the dlt motivic zeta function of $f$ with respect to the dlt modification $g$ by
		\[
		Z^{dlt}_{f}((X^{dlt},Z(f)^{dlt}),s)=\L^{-\dim(X)}\sum\limits_{I\subseteq S} \mathcal{E}_{st}(\mathring{E}_I,\mathring{D}_I)\prod\limits_{i\in I}\frac{\mathbb{L}-1}{\mathbb{L}^{N_is+\nu_i}-1}\in \tilde{\mathcal{M}}_k^e[[\L^{-s}]]
		\]
		where $\tilde{\mathcal{M}}_k^e$ is the same ring as in Definition \ref{def:stringy}.

	\section{Counterexamples}
	\label{counterexample}
	Consider the polynomial $f=x^4+x^2y^2+y^6+z^3\in\C[x,y,z]$. The surface cut out by $f$ has an isolated singularity at the origin. We will consider two dlt modifications $$g_i\colon (X_i,\Delta_i)\rightarrow (\A^3,Z(f)),\,i\in\{1,2\}$$ that are connected via a blow-up, and show that $Z^{dlt}_{f}((X_1,\Delta_1),s)\neq Z^{dlt}_{f}((X_2,\Delta_2),s)$. We will also construct two dlt modifications $$g_j\colon (X_j,\Delta_j)\rightarrow (\A^3,Z(f)),j\in\{A,B\}$$ that are connected via a flop, and show that $Z^{dlt}_{f}((X_A,\Delta_A),s)\neq Z^{dlt}_{f}((X_B,\Delta_B),s)$. These Examples contradict Propositions 3.2 and 3.1 in \cite{Xu1}, respectively (see Section \ref{sec:Xu}). In particular, they show that the dlt motivic zeta function depends on the choice of the dlt modification, contradicting Theorem 1.2 in \cite{Xu1}.
	
	The polynomial $f$ is non-degenerate at the origin (see Definition 1.19 of \cite{Kou}). Therefore, we can use toric methods to construct the relevant dlt-modifications.
	We start by fixing a few notational conventions.

		\begin{enumerate}
			\item[i.] Given a fan $\Sigma$, we denote the associated toric variety by $X(\Sigma)$.
			\item[ii.] We denote  by $\Sigma^3$ the standard fan of $\A^3$, so that $X(\Sigma^3) = \A^3$.
			\item[iii.] Consider $u_1,...,u_k \in \R^3$. Define
			\[
			\langle u_1,...,u_k \rangle_{\geq 0}=\left\{\sum_{i=1}^k\lambda_iu_i\mid \lambda_i\in\R_{\geq 0}\right\}.
			\]
			\item[iv.] We denote the standard inner product on $\R^3$ by $\langle \cdot , \cdot \rangle$.
			\item[v.] We denote the primitive generator of a ray $\rho_i$ by $u_i$.
			\item[vi.] Consider a cone $\sigma\in\Sigma$. We denote the open orbit of $\sigma$ by $O(\sigma)$ and the closed orbit by $V(\sigma)$. If $\rho_i$ is a ray, we set $E_i=V(\rho_i)$. If $\rho_1,...,\rho_k$ are rays, we denote $O(\langle u_1,...,u_k\rangle_{\geq 0})$ by $O(\rho_1,...,\rho_k)$ and $V(\langle u_1,...,u_k\rangle_{\geq 0})$ by $V(\rho_1,...,\rho_k)$.
			\item[vii.] Let $\pi:Y\rightarrow X$ be a birational morphism and $E\subset X$ a divisor. We denote the birational transform of $E$ by $\tilde{E}\subset Y$.
		\end{enumerate}
		
	\subsection{Example 1: dlt modifications connected via a blow-up}
	\label{blowup}
	Consider the fan $\Sigma_1$ as depicted in Figure $\ref{dltmodel1}$ (the picture shows the intersection of the fan with the plane described by $x+y+z=1$).
	
	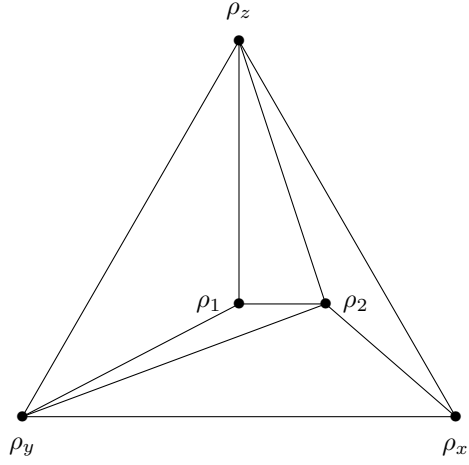
\begin{figure}[h!]
		\centering
		\begin{tikzpicture}[scale=2]
			\coordinate (ux) at (2.88/2,5/4);
			\coordinate (uy) at (-2.88/2,5/4);
			\coordinate (uz) at (0,15/4);
			\coordinate (u1) at (3*2.88/10-3*2.88/10,4*5/10); 
			\coordinate (u2) at (2*2.88/5-1*2.88/5,2*5/5); 
			\coordinate (u4) at (1*2.88/4-1*2.88/4,2*5/4); 
			\coordinate (u5) at (2*2.88/7-2*2.88/7,3*5/7); 
			\coordinate (u6) at (1*2.88/3-1*2.88/3,1*5/3); 
			
			\draw (u1) -- (u2);
			\draw (u1) -- (uy);
			\draw (u1) -- (uz);
			\draw (u2) -- (uz);
			\draw (u2) -- (ux);
			\draw (u2) -- (uy);
			\draw (uz) -- (ux);
			\draw (uz) -- (uy);
			\draw (ux) -- (uy);
			
			\node at (2.88/2,5/4) {$\small\bullet$};
			\node at (-2.88/2,5/4) {$\small\bullet$};			
			\node at (0,15/4) {$\small\bullet$};				
			\node at (3*2.88/10-3*2.88/10,4*5/10) {$\small\bullet$};
			\node at (2*2.88/5-1*2.88/5,2*5/5) {$\small\bullet$};

			\node at (2.88/2,5/4-0.2) {$\rho_x$};
			\node at (-2.88/2,5/4-0.2){$\rho_y$};			
			\node at (0,15/4+0.2) {$\rho_z$};				
			\node at (3*2.88/10-3*2.88/10-0.2,4*5/10) {$\rho_1$};
			\node at (2*2.88/5-1*2.88/5+0.2,2*5/5) {$\rho_2$};
		\end{tikzpicture}
		\caption{The fan $\Sigma_1$}
		\label{dltmodel1}
	\end{figure}
	
	Here, the set of rays $\Sigma_1(1)$ of $\Sigma_1$ consists of
 \begin{align*}
 &\rho_x=\langle(1,0,0)\rangle_{\geq 0},\qquad\rho_y=\langle(0,1,0)\rangle_{\geq 0},\qquad\rho_z=\langle(0,0,1)\rangle_{\geq 0},\\
 &\rho_1=\langle(3,3,4)\rangle_{\geq 0},\qquad\rho_2=\langle(2,1,2)\rangle_{\geq 0}.
 \end{align*}
 There is the natural morphism $g_1\colon X_1=X(\Sigma_1)\rightarrow \A^3$ that is induced by the refinement of fans $\Sigma^3\subset \Sigma_1$. Furthermore, $(N_1,\nu_1)=(12,10)$ and $(N_2,\nu_2)=(6,5)$.
	
	\begin{prop}
	The morphism $g_1\colon (X_1,\Delta_1)\rightarrow (\A^3,Z(f))$ is a dlt modification.
	\end{prop}
	\begin{proof}
		First, we show that the pair $(X_1,\Delta_1)$ is dlt. Consider the fan $\Sigma^{res}$ in Figure $\ref{resolution}$. Here, $u_4=(1,1,2)$, $u_5=(2,2,3)$ and $u_6=(1,1,1)$. We compute that $(N_4,\nu_4)=(4,4)$, $(N_5,\nu_5)=(8,7)$ and $(N_6,\nu_6)=(3,3)$.
		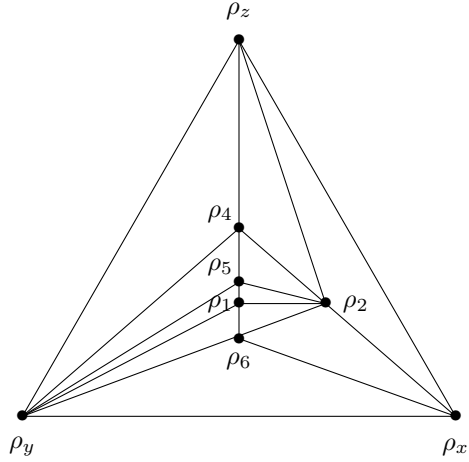
\begin{figure}[h!]
			\centering
			\begin{tikzpicture}[scale=2]
				\coordinate (ux) at (2.88/2,5/4);
				\coordinate (uy) at (-2.88/2,5/4);
				\coordinate (uz) at (0,15/4);
				\coordinate (u1) at (3*2.88/10-3*2.88/10,4*5/10); 
				\coordinate (u2) at (2*2.88/5-1*2.88/5,2*5/5); 
				\coordinate (u4) at (1*2.88/4-1*2.88/4,2*5/4); 
				\coordinate (u5) at (2*2.88/7-2*2.88/7,3*5/7); 
				\coordinate (u6) at (1*2.88/3-1*2.88/3,1*5/3+0.1); 
				
				\draw (u1) -- (u2);
				\draw (u1) -- (uy);
				\draw (u1) -- (uz);
				\draw (u2) -- (uz);
				\draw (u2) -- (ux);
				\draw (u2) -- (uy);
				\draw (uz) -- (ux);
				\draw (uz) -- (uy);
				\draw (ux) -- (uy);
				\draw (u4) -- (u2);
				\draw (u5) -- (u2);
				\draw (u1) -- (u6);
				\draw (u6) -- (ux);
				\draw (uy) -- (u4);
				\draw (uy) -- (u5);
				
				\node at (2.88/2,5/4) {$\small\bullet$};
				\node at (-2.88/2,5/4) {$\small\bullet$};			
				\node at (0,15/4) {$\small\bullet$};				
				\node at (3*2.88/10-3*2.88/10,4*5/10) {$\small\bullet$};
				\node at (2*2.88/5-1*2.88/5,2*5/5) {$\small\bullet$};
				\node at (1*2.88/4-1*2.88/4,2*5/4) {$\small\bullet$}; 
				\node at (2*2.88/7-2*2.88/7,3*5/7) {$\small\bullet$}; 
				\node at (1*2.88/3-1*2.88/3,1*5/3+0.1) {$\small\bullet$}; 
				
				\node at (2.88/2,5/4-0.2) {$\rho_x$};
				\node at (-2.88/2,5/4-0.2){$\rho_y$};			
				\node at (0,15/4+0.2) {$\rho_z$};				
				\node at (3*2.88/10-3*2.88/10-0.125,4*5/10) {$\rho_1$};
				\node at (2*2.88/5-1*2.88/5+0.2,2*5/5) {$\rho_2$};
				\node at (1*2.88/4-1*2.88/4-0.125,2*5/4+0.1) {$\rho_4$}; 
				\node at (2*2.88/7-2*2.88/7-0.125,3*5/7+0.1) {$\rho_5$}; 
				\node at (1*2.88/3-1*2.88/3,1*5/3-0.25+0.2) {$\rho_6$}; 
			\end{tikzpicture}
			\caption{The fan $\Sigma^{res}$}
			\label{resolution}
		\end{figure}
		The induced morphism $\pi\colon Y=X(\Sigma^{res})\rightarrow X_1$ is a log resolution of $(X_1,\Delta_1)$ because $\Sigma^{res}$ is regular (see Section 9 and 10 of \cite{Va}).
Denote the birational transform of $Z(f)$ on $X_1$ by $E_0$, and on $X_{res}$ by $\tilde{E_0}$, respectively. Then
\[K_{X_1}+\Delta_1\equiv\nu_1 E_1+ \nu_2 E_2 + E_0 \equiv (\nu_1-N_1)E_1+(\nu_2-N_2)E_2 = -2E_1-E_2=:F_1, \]
since $N_1E_1+N_2E_2+ E_0$ is linearly equivalent to the zero divisor.

Analogously, we have
\begin{eqnarray*}
	&&\nu_1E_1+\nu_2E_2+\nu_4E_4+\nu_5E_5+\nu_6E_6 + \tilde{E_0} \\
	&\equiv& (\nu_1-N_1)E_1+(\nu_2-N_2)E_2+(\nu_4-N_4)E_4+(\nu_5-N_5)E_5+(\nu_6-N_6)E_6\\
	&=&-2E_1-E_2-E_5=:F_{res}.
\end{eqnarray*}	

		Consider the support function $\phi_1:|\Sigma_1|\rightarrow\R$ that is associated to $F_1$, that is, the function that satisfies
		\[
		\phi_1(u_x)=\phi_1(u_y)=\phi_1(u_z)=0, \phi_1(u_1)=2,  \phi_1(u_2)=1
		\]
		and that is linear on the cones of $\Sigma_1$. Similarly, let $\phi_{res}:|\Sigma_{res}|\rightarrow\R$ be the support function corresponding to $F_{res}$.
	 			Then we have
		\begin{eqnarray*}
			& &K_{X(\Sigma^{res})}+\tilde{\Delta}_1+\mathrm{Exc}(\pi)-\pi^*(K_{X_1}+\Delta_1)\\
			&=&F_{res}-\pi^*(F_1)\\
			&=&\sum_{\substack{i=1\\i\neq 3}}^6\phi_{res}(u_i)E_i-\sum_{\substack{i=1\\i\neq 3}}^6\phi_1(u_i)E_i\\
			&=&\frac{2}{3}E_4+\frac{1}{3}E_5+\frac{1}{2}E_6.
		\end{eqnarray*}
		Hence, the log discrepancies of $E_4$, $E_5$ and $E_6$ are $a_4=\frac{2}{3}$, $a_5=\frac{1}{3}$ and $a_6=\frac{1}{2}$, respectively.
		
	\smallskip	
		Finally, consider the closed set $$C=V(\rho_y,\rho_2)\cup V(\rho_z,\rho_1)\subset X_1$$ and write $U=X_1\setminus C$. Then $U$ is a open set such that $(U,\Delta_1\vert_U)$ is snc since the corresponding cones are regular. Furthermore, no log canonical center of $(X_1,\Delta_1)$ is contained in $X_1\setminus U=C$. Therefore, the pair $(X_1,\Delta_1)$ is dlt.
		
Next, we show that the divisor $K_{X_1}+\Delta_1$ is $g_1$-nef. By Theorem 6.1.7 of \cite{CLS}, it suffices to show that $\phi_1$ is convex.
		For each maximal cone $\sigma$ of $\Sigma$, define $m_\sigma\in\R^3$ such that $\phi_1(x)=\langle m_\sigma,x\rangle$ for all $x\in \sigma$. By Lemma 6.1.5 of \cite{CLS}, it is enough to show that for every wall $\tau=\sigma\cap\sigma'$, where $\sigma$ and $\sigma'$ are maximal cones of $\Sigma_1$, there is some $u_0\in\sigma'\setminus\sigma$ with $\phi_1(u_0)\leq \langle m_\sigma,u_0 \rangle$.
		We introduce the following notation:
				\[\begin{array}{lllclllclll}
			\sigma_1&=&\langle \rho_x,\rho_y,\rho_2 \rangle_{\geq 0},&  &
			\sigma_2&=&\langle \rho_x,\rho_z,\rho_2 \rangle_{\geq 0},&  &
			\sigma_3&=&\langle \rho_y,\rho_z,\rho_1 \rangle_{\geq 0}, \\[1.5ex]
			\sigma_4&=&\langle \rho_y,\rho_1,\rho_2 \rangle_{\geq 0},&  &
			\sigma_5&=&\langle \rho_z,\rho_1,\rho_2 \rangle_{\geq 0}. & & & &
		\end{array}\]
		One computes that
		\[\begin{array}{lllclllclll}
		m_{\sigma_1}&=& (0,0,\frac{1}{2}),&\qquad&
		m_{\sigma_2}&=&(0,1,0),&\qquad&
		m_{\sigma_3}&=&(\frac{2}{3},0,0) ,\\[1.5ex]
		m_{\sigma_4}&=&(0,0,\frac{1}{2}),&\qquad&
		m_{\sigma_5}&=&(\frac{1}{3},\frac{1}{3},0).& & & &
		\end{array}\]
		In this example, there are six walls, namely,
				\[\begin{array}{lllclllclll}
		\tau_1&=& \sigma_1 \cap \sigma_2,&\quad &
		\tau_2&=&\sigma_1 \cap \sigma_4,&\quad &
		\tau_3&=&\sigma_2 \cap \sigma_5 ,\\
		\tau_4&=&\sigma_3 \cap \sigma_4, &\quad &
		\tau_5&=&\sigma_3 \cap \sigma_5, &\quad &
		\tau_6&=&\sigma_4 \cap \sigma_5.
		\end{array}\]
		Consider the wall $\tau_1$ and consider any point $u_0$ in $\sigma_2\setminus \sigma_1$, e.g. $u_0=u_x+u_z+u_2$. Compute that $\phi_1(u_0)=0+0+1=1$ and
		\[
		\langle m_{\sigma_1},u_0 \rangle = \langle (0,0,\frac{1}{2}),(3,1,3) \rangle = \frac{3}{2}\geq 1.
		\]			
		The other wall-inequalities are verified in the same way. This shows that $K_{X_1}+\Delta_1$ is $g_1$-nef. Therefore, $g_1$ is indeed a dlt modification.
	\end{proof}
	
	To construct our second dlt-modification, we consider the fan $\Sigma_2$ as depicted in Figure $\ref{dltmodel2}$.
	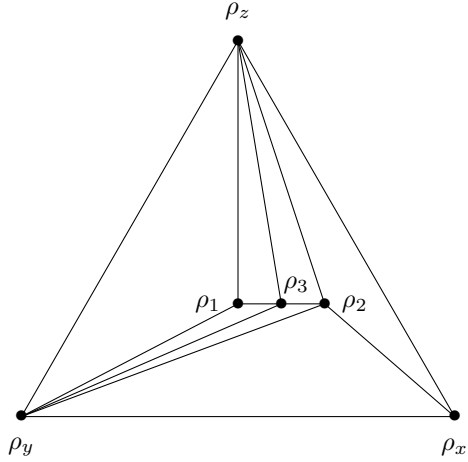
\begin{figure}[h!]
		\centering
		\begin{tikzpicture}[scale=2]
			\coordinate (ux) at (2.88/2,5/4);
			\coordinate (uy) at (-2.88/2,5/4);
			\coordinate (uz) at (0,15/4);
			\coordinate (u1) at (3*2.88/10-3*2.88/10,4*5/10); 
			\coordinate (u2) at (2*2.88/5-1*2.88/5,2*5/5); 
			\coordinate (u3) at (3*2.88/20-3*2.88/20+2*2.88/10-1*2.88/10,4*5/20+2*5/10); 
			\coordinate (u4) at (1*2.88/4-1*2.88/4,2*5/4); 
			\coordinate (u5) at (2*2.88/7-2*2.88/7,3*5/7); 
			\coordinate (u6) at (1*2.88/3-1*2.88/3,1*5/3); 
			
			\draw (u1) -- (u2);
			\draw (u1) -- (uy);
			\draw (u1) -- (uz);
			\draw (u2) -- (uz);
			\draw (u2) -- (ux);
			\draw (u2) -- (uy);
			\draw (uz) -- (ux);
			\draw (uz) -- (uy);
			\draw (ux) -- (uy);
			\draw (u3) -- (uy);
			\draw (u3) -- (uz);
			
			\node at (2.88/2,5/4) {$\small\bullet$};
			\node at (-2.88/2,5/4) {$\small\bullet$};			
			\node at (0,15/4) {$\small\bullet$};		
			\node at (3*2.88/10-3*2.88/10,4*5/10) {$\small\bullet$};
			\node at (2*2.88/5-1*2.88/5,2*5/5) {$\small\bullet$};
			\node at (3*2.88/20-3*2.88/20+2*2.88/10-1*2.88/10,4*5/20+2*5/10) {$\small\bullet$}; 
	
			\node at (2.88/2,5/4-0.2) {$\rho_x$};
			\node at (-2.88/2,5/4-0.2){$\rho_y$};			
			\node at (0,15/4+0.2) {$\rho_z$};					
			\node at (3*2.88/10-3*2.88/10-0.2,4*5/10) {$\rho_1$};
			\node at (2*2.88/5-1*2.88/5+0.2,2*5/5) {$\rho_2$};
			\node at (3*2.88/20-3*2.88/20+2*2.88/10-1*2.88/10+0.1,4*5/20+2*5/10+0.125) {$\rho_3$}; 
		\end{tikzpicture}
		\caption{The fan $\Sigma_2$}
		\label{dltmodel2}
	\end{figure}
	
	Remark that $$\Sigma_2(1)=\Sigma_1(1)\cup\{\rho_3=\langle (5,4,6) \rangle_{\geq 0}\}.$$ Furthermore, there is a natural morphism \[\gamma\colon X_2=X(\Sigma_2)\rightarrow X_1\] that is induced by the refinement $\Sigma_1\subset\Sigma_2$. If we restrict to the open set $U$ from before, the map $\gamma$ is the blow-up of $U$ in the centre $V(\rho_1,\rho_2)\cap U$. Define $g_2=\gamma\circ g_1: X_2\rightarrow \A^3$. This is the toric morphism induced by the refinement $\Sigma^3\subset \Sigma_2$.
	
	\begin{prop}
	The morphism $g_2\colon (X_2,\Delta_2)\rightarrow (\A^3,Z(f))$ is a dlt modification.
	\end{prop}\begin{proof}
	It is easy to check that $(X_2,\Delta_2)$ is a dlt pair. Furthermore, $K_{X_2}+\Delta_2$ is $g_2$-nef since the associated support function $\phi_2$ is equal to $\phi_1$ and therefore convex. Indeed, $\phi_2(u_3)=N_3-\nu_3=3$ by construction and it is clear that $\phi_1(u_3)=3$ as well.
\end{proof}	
	
 We now compute the dlt zeta functions of $f$ associated to each of our dlt modifications.
	The morphism
	\[
	\gamma\vert_{X_2\setminus E_3}\colon X_2\setminus E_3\rightarrow X_1\setminus (E_1\cap E_2)
	\]
	is an isomorphism.
	Therefore, we only need to consider the contribution of strata that are contained in $E_3$ on the level of $X_2$ and strata that are contained in $E_1\cap E_2$ on the level of $X_1$.
	First, we focus on $Z_{f}^{dlt}((X_1,\Delta_1),s)$. The relevant part is
	\begin{align*}
		Z_1(s)&=\L^{-3}  \Bigl( \mathcal{E}_{st}(\mathring{E}_{1,2},\mathring{D}_{1,2})\frac{\mathbb{L}-1}{\mathbb{L}^{N_1s+\nu_1}-1}\frac{\mathbb{L}-1}{\mathbb{L}^{N_2s+\nu_2}-1}\\&+\mathcal{E}_{st}(\mathring{E}_{0,1,2},\mathring{D}_{0,1,2})\frac{\mathbb{L}-1}{\mathbb{L}^{N_0s+\nu_0}-1}\frac{\mathbb{L}-1}{\mathbb{L}^{N_1s+\nu_1}-1}\frac{\mathbb{L}-1}{\mathbb{L}^{N_2s+\nu_2}-1}\Bigr).
	\end{align*}
	Here, $E_0$ stands for the birational transform of $Z(f)$ in $X_1$. We have $(N_1,\nu_1)=(12,10)$, $(N_2,\nu_2)=(6,5)$ and $(N_0,\nu_0)=(1,1)$.
	We also observe that
	\[
	\mathring{E}_{1,2}=(O(\rho_1,\rho_2)\setminus E_0)\cup O(\rho_1,\rho_2,\rho_y)\cup O(\rho_1,\rho_2,\rho_z).
	\]
	For the log resolution $\pi\colon Y\rightarrow X$ constructed above, we
	 obtain that
	\begin{align*}
		D_{1,2}&=(\pi\vert_{\tilde{E}_1\cap\tilde{E}_2})_*(((\pi^*(K_{X_1}+\Delta_1)-K_Y-\tilde{E}_1-\tilde{E}_2))\vert_{\tilde{E}_1\cap \tilde{E}_2})\\
		&=(\pi\vert_{\tilde{E}_1\cap\tilde{E}_2})_*((\tilde{E}_0-(a_4-1)E_4-(a_5-1)E_5-(a_6-1)E_6))\vert_{\tilde{E}_1\cap \tilde{E}_2})\\
		&=O(\rho_1,\rho_2)\cap E_0+(1-a_6)\cdot O(\rho_1,\rho_2,\rho_y)+(1-a_5)\cdot O(\rho_1,\rho_2,\rho_z)\\
		&=O(\rho_1,\rho_2)\cap E_0+\frac{1}{2}O(\rho_1,\rho_2,\rho_y) + \frac{2}{3} O(\rho_1,\rho_2,\rho_z).
	\end{align*}
	Therefore,
	\[
	\mathring{D}_{1,2}=D_{1,2}\vert_{\mathring{E}_{1,2}}=\frac{1}{2}O(\rho_1,\rho_2,\rho_y) + \frac{2}{3} O(\rho_1,\rho_2,\rho_z).
	\]
The intersection $O(\rho_1,\rho_2)\cap E_0$ consists of a single point, so that $D_{0,1,2}=0$ and $[O(\rho_1,\rho_2)\cap E_0]=1.$
	Combining all the above, we compute
	\[
	\mathcal{E}_{st}(\mathring{E}_{1,2},\mathring{D}_{1,2})=\L-2+\frac{\L-1}{\L^{\nicefrac{1}{2}}-1}+\frac{\L-1}{\L^{\nicefrac{1}{3}}-1} \quad \mbox{and} \quad \mathcal{E}_{st}(\mathring{E}_{0,1,2},\mathring{D}_{0,1,2})=1.
	\]
	Therefore, we obtain
	\begin{align*}
		Z_1(s)&=\L^{-3}  \Bigl( \Bigl(\L-2+\frac{\L-1}{\L^{\nicefrac{1}{2}}-1}+\frac{\L-1}{\L^{\nicefrac{1}{3}}-1}\Bigr)\cdot\frac{\mathbb{L}-1}{\mathbb{L}^{N_1s+\nu_1}-1}\frac{\mathbb{L}-1}{\mathbb{L}^{N_2s+\nu_2}-1}\\&+1\cdot\frac{\mathbb{L}-1}{\mathbb{L}^{N_0s+\nu_0}-1}\frac{\mathbb{L}-1}{\mathbb{L}^{N_1s+\nu_1}-1}\frac{\mathbb{L}-1}{\mathbb{L}^{N_2s+\nu_2}-1}\Bigr).
	\end{align*}
	
	Next, we compute the relevant part $Z_2(s)$ of $Z^{dlt}_{f}((X_2,D_2),s)$, that is,
	\begin{align*}
		Z_2(s)&=\L^{-3} \Bigl( \mathcal{E}_{st}(\mathring{E}_{3},\mathring{D}_{3})\frac{\mathbb{L}-1}{\mathbb{L}^{N_3s+\nu_3}-1}+\mathcal{E}_{st}(\mathring{E_{0,3}},\mathring{D}_{0,3})\frac{\mathbb{L}-1}{\mathbb{L}^{N_0s+\nu_0}-1}\frac{\mathbb{L}-1}{\mathbb{L}^{N_3s+\nu_3}-1}\\
		&+\mathcal{E}_{st}(\mathring{E}_{1,3},\mathring{D}_{1,3})\frac{\mathbb{L}-1}{\mathbb{L}^{N_1s+\nu_1}-1}\frac{\mathbb{L}-1}{\mathbb{L}^{N_3s+\nu_3}-1}\\
		&+\mathcal{E}_{st}(\mathring{E}_{0,1,3},\mathring{D}_{0,1,3})\frac{\mathbb{L}-1}{\mathbb{L}^{N_0s+\nu_0}-1}\frac{\mathbb{L}-1}{\mathbb{L}^{N_1s+\nu_1}-1}\frac{\mathbb{L}-1}{\mathbb{L}^{N_3s+\nu_3}-1}\\
		&+\mathcal{E}_{st}(\mathring{E}_{2,3},\mathring{D}_{2,3})\frac{\mathbb{L}-1}{\mathbb{L}^{N_2s+\nu_2}-1}\frac{\mathbb{L}-1}{\mathbb{L}^{N_3s+\nu_3}-1}\\
		&+\mathcal{E}_{st}(\mathring{E}_{0,2,3},\mathring{D}_{0,2,3})\frac{\mathbb{L}-1}{\mathbb{L}^{N_0s+\nu_0}-1}\frac{\mathbb{L}-1}{\mathbb{L}^{N_2s+\nu_2}-1}\frac{\mathbb{L}-1}{\mathbb{L}^{N_3s+\nu_3}-1}\Bigr) .
	\end{align*}
	Remark that $(N_3,\nu_3)=(18,15)$. Observe that $$[\mathring{E}_{0,3}]=[O(\rho_3)\cap E_0]=\L-1.$$ Next, observe that $\mathring{E}_3=(O(\rho_3)\setminus E_0) \cup O(\rho_3,\rho_y)\cup O(\rho_3,\rho_z)$ and $\mathring{D}_3=D_3\vert_{\mathring{E}_3}=0.$ Therefore, we compute
	\[
	\mathcal{E}_{st}(\mathring{E}_3,\mathring{D}_3)=(\L-1)^2+\L-1 \quad \mbox{and} \quad \mathcal{E}_{st}(\mathring{E}_{0,3},\mathring{D}_{0,3})=\L-1.
	\]
	Similarly to the computations on the level of $X_1$, one computes
	\begin{align*}
		\mathcal{E}_{st}(\mathring{E}_{1,3},\mathring{D}_{1,3})=\L-2 +\frac{\L-1}{\L^{\nicefrac{1}{2}}-1}+\frac{\L-1}{\L^{\nicefrac{1}{3}}-1}, \qquad \mathcal{E}_{st}(\mathring{E}_{0,1,3},\mathring{D}_{0,1,3})=1,\\
		\mathcal{E}_{st}(\mathring{E}_{2,3},\mathring{D}_{2,3})=\L-2 +\frac{\L-1}{\L^{\nicefrac{1}{2}}-1}+\frac{\L-1}{\L^{\nicefrac{1}{3}}-1}, \qquad \mathcal{E}_{st}(\mathring{E}_{0,2,3},\mathring{D}_{0,2,3})=1.
	\end{align*}
	
	Combining all of this, we obtain
	\begin{align*}
		Z_2(s)&=\L^{-3}  \Bigl( ((\L-1)^2+\L-1)\cdot\frac{\mathbb{L}-1}{\mathbb{L}^{N_3s+\nu_3}-1}
		\\&+(\L-1)\cdot\frac{\mathbb{L}-1}{\mathbb{L}^{N_0s+\nu_0}-1}\frac{\mathbb{L}-1}{\mathbb{L}^{N_3s+\nu_3}-1}\\
		&+\Bigl(\L-2 +\frac{\L-1}{\L^{\nicefrac{1}{2}}-1}+\frac{\L-1}{\L^{\nicefrac{1}{3}}-1}\Bigr)\cdot \frac{\mathbb{L}-1}{\mathbb{L}^{N_1s+\nu_1}-1}\frac{\mathbb{L}-1}{\mathbb{L}^{N_3s+\nu_3}-1}\\
		&+1\cdot\frac{\mathbb{L}-1}{\mathbb{L}^{N_0s+\nu_0}-1}\frac{\mathbb{L}-1}{\mathbb{L}^{N_1s+\nu_1}-1}\frac{\mathbb{L}-1}{\mathbb{L}^{N_3s+\nu_3}-1}\\
		&+\Bigl(\L-2 +\frac{\L-1}{\L^{\nicefrac{1}{2}}-1}+\frac{\L-1}{\L^{\nicefrac{1}{3}}-1}\Bigr)\cdot\frac{\mathbb{L}-1}{\mathbb{L}^{N_2s+\nu_2}-1}\frac{\mathbb{L}-1}{\mathbb{L}^{N_3s+\nu_3}-1}\\
		&+1\cdot\frac{\mathbb{L}-1}{\mathbb{L}^{N_0s+\nu_0}-1}\frac{\mathbb{L}-1}{\mathbb{L}^{N_2s+\nu_2}-1}\frac{\mathbb{L}-1}{\mathbb{L}^{N_3s+\nu_3}-1}\Bigr).
	\end{align*}
	
	It is easy to verify the equality
	\begin{align*}
\frac{\L-1}{\L^{N_1s+\nu_1}-1}\frac{\L-1}{\L^{N_2s+\nu_2}-1}&=\frac{\L-1}{\L^{N_1s+\nu_1}-1}\frac{\L-1}{\L^{N_3s+\nu_3}-1}\\&+\frac{\L-1}{\L^{N_2s+\nu_2}-1}\frac{\L-1}{\L^{N_3s+\nu_3}-1}+\frac{(\L-1)^2}{\L^{N_3s+\nu_3}-1}.
	\end{align*}
	Now a simple computation yields
	\begin{eqnarray*}
	Z_{f}^{dlt}((X_1,D_1),s)- Z_{f}^{dlt}((X_2,D_2),s)&=&Z_1(s)-Z_2(s)
	\\&=&\Bigl(\frac{\L-1}{\L^{\nicefrac{1}{2}}-1}+\frac{\L-1}{\L^{\nicefrac{1}{3}}-1}-2 \Bigr) \frac{(\L-1)^2}{\L^{N_3s+\nu_3}-1}
	\end{eqnarray*}
	and this expression is different from zero. It follows that $Z_{f}^{dlt}((X_1,D_1),s)\neq Z_{f}^{dlt}((X_2,D_2),s)$.
	
	\subsection{Example 2: dlt modifications connected via a flop}
	\label{flop}
	Consider the fans $\Sigma_A=\Sigma_1$ and $\Sigma_B$ as depicted in Figure \ref{dltmodelsAB}.
	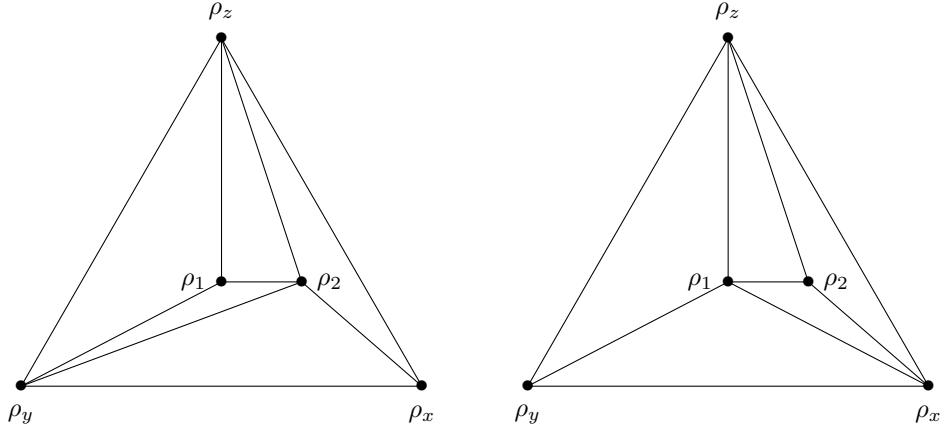
\begin{figure}[h!]
		\centering
		\begin{tikzpicture}[scale=1.85]
			\coordinate (ux) at (2.88/2,5/4);
			\coordinate (uy) at (-2.88/2,5/4);
			\coordinate (uz) at (0,15/4);
			\coordinate (u1) at (3*2.88/10-3*2.88/10,4*5/10); 
			\coordinate (u2) at (2*2.88/5-1*2.88/5,2*5/5); 
			\coordinate (u4) at (1*2.88/4-1*2.88/4,2*5/4); 
			\coordinate (u5) at (2*2.88/7-2*2.88/7,3*5/7); 
			\coordinate (u6) at (1*2.88/3-1*2.88/3,1*5/3); 
			
			\draw (u1) -- (u2);
			\draw (u1) -- (uy);
			\draw (u1) -- (uz);
			\draw (u2) -- (uz);
			\draw (u2) -- (ux);
			\draw (u2) -- (uy);
			\draw (uz) -- (ux);
			\draw (uz) -- (uy);
			\draw (ux) -- (uy);
			
			\node at (2.88/2,5/4) {$\small\bullet$};
			\node at (-2.88/2,5/4) {$\small\bullet$};			
			\node at (0,15/4) {$\small\bullet$};				
			\node at (3*2.88/10-3*2.88/10,4*5/10) {$\small\bullet$};
			\node at (2*2.88/5-1*2.88/5,2*5/5) {$\small\bullet$};

			\node at (2.88/2,5/4-0.2) {$\rho_x$};
			\node at (-2.88/2,5/4-0.2){$\rho_y$};			
			\node at (0,15/4+0.2) {$\rho_z$};				
			\node at (3*2.88/10-3*2.88/10-0.2,4*5/10) {$\rho_1$};
			\node at (2*2.88/5-1*2.88/5+0.2,2*5/5) {$\rho_2$};
		\end{tikzpicture}
		\qquad
		\begin{tikzpicture}[scale=1.85]
			\coordinate (ux) at (2.88/2,5/4);
			\coordinate (uy) at (-2.88/2,5/4);
			\coordinate (uz) at (0,15/4);
			\coordinate (u1) at (3*2.88/10-3*2.88/10,4*5/10); 
			\coordinate (u2) at (2*2.88/5-1*2.88/5,2*5/5); 
			\coordinate (u4) at (1*2.88/4-1*2.88/4,2*5/4); 
			\coordinate (u5) at (2*2.88/7-2*2.88/7,3*5/7); 
			\coordinate (u6) at (1*2.88/3-1*2.88/3,1*5/3); 
			
			\draw (u1) -- (u2);
			\draw (u1) -- (uy);
			\draw (u1) -- (uz);
			\draw (u2) -- (uz);
			\draw (u2) -- (ux);
			\draw (u1) -- (ux);
			\draw (uz) -- (ux);
			\draw (uz) -- (uy);
			\draw (ux) -- (uy);
			
			\node at (2.88/2,5/4) {$\small\bullet$};
			\node at (-2.88/2,5/4) {$\small\bullet$};			
			\node at (0,15/4) {$\small\bullet$};		
			\node at (3*2.88/10-3*2.88/10,4*5/10) {$\small\bullet$};
			\node at (2*2.88/5-1*2.88/5,2*5/5) {$\small\bullet$};

			\node at (2.88/2,5/4-0.2) {$\rho_x$};
			\node at (-2.88/2,5/4-0.2){$\rho_y$};			
			\node at (0,15/4+0.2) {$\rho_z$};			
			\node at (3*2.88/10-3*2.88/10-0.2,4*5/10) {$\rho_1$};
			\node at (2*2.88/5-1*2.88/5+0.2,2*5/5) {$\rho_2$};
		\end{tikzpicture}
		\caption{The fans $\Sigma_A$ and $\Sigma_B$}
		\label{dltmodelsAB}
	\end{figure}
	Define $X_A=X(\Sigma_A)$ and $X_B=X(\Sigma_B)$. The induced morphisms $g_A:(X_A,\Delta_A)\rightarrow (\A^3,Z(f))$ and $g_B:(X_B,\Delta_B)\rightarrow (\A^3,Z(f))$ are dlt modifications. Indeed, $g_A=g_1$, so we already verified that this morphism is a dlt modification in Section \ref{blowup}. Showing that $(X_B,\Delta_B)$ is a dlt pair can be verified in a similar way as in Section \ref{blowup}. The fact that $K_{X_B}+\Delta_B$ is $g_B$-nef follows immediately from our computations from before: the support function $\phi_1$ is linear on the cone $\langle \rho_x,\rho_y,\rho_1,\rho_2 \rangle_{\geq 0}$. In particular, $\phi_1=\phi_A=\phi_B$ which implies that $\phi_B$ is convex.

	Denote by $Z_A(s)$ the contribution to $Z_{f}^{dlt}((X_A,\Delta_A),s)$ that corresponds to $V(\rho_y,\rho_2)$  and denote by $Z_B(s)$ the contribution to $Z_{f}^{dlt}((X_B,\Delta_B),s)$ that corresponds to $V(\rho_x,\rho_1)$. The restriction
	\[
	\eta\vert_{X_A\setminus V(\rho_y,\rho_2)}:X_A\setminus V(\rho_y,\rho_2)\rightarrow X_B\setminus V(\rho_x,\rho_1)
	\]
	is an isomorphism. Therefore, we see that $Z_A(s)$ and $Z_B(s)$ are the only contributions of $Z_{f}^{dlt}((X_A,\Delta_A),s)$ and $Z_{f}^{dlt}((X_B,\Delta_B),s)$, respectively, that could possibly differ.
	Using the same techniques as in Section \ref{blowup} one computes
	\[
	Z_A(s)=\Bigl( (\L-1)\frac{\L-1}{\L^{\nicefrac{1}{2}}-1}+\frac{\L-1}{\L^{\nicefrac{1}{2}}-1} \ \Bigr)\frac{\L-1}{\L^{N_2s+\nu_2}-1}+\frac{\L-1}{\L^{\nicefrac{1}{2}}-1}\frac{\L-1}{\L^{N_1s+\nu_1}-1}\frac{\L-1}{\L^{N_2s+\nu_2}-1}
	\]
	and
	\[
	Z_B(s)=\Bigl( \L-1 + (\L+1)\frac{\L-1}{\L^{\nicefrac{1}{2}}-1} \ \Bigr)\frac{\L-1}{\L^{N_1s+\nu_1}-1}+\frac{\L-1}{\L^{\nicefrac{1}{2}}-1}\frac{\L-1}{\L^{N_1s+\nu_1}-1}\frac{\L-1}{\L^{N_2s+\nu_2}-1}.
	\]
	Once again, it is not hard to verify that $Z_A(s)\neq Z_B(s)$, so that $Z_{f}^{dlt}((X_A,D_A),s)\neq Z_{f}^{dlt}((X_B,D_B),s)$.
	
	\section{On Xu's proof}\label{sec:Xu}
	In this section, we will outline the problems in the proof of Xu's claim that $Z_{f}^{dlt}$ is independent of the choice of dlt modification.
		Let $f$ be a non-constant regular function on a smooth quasi-projective $k$-variety $X$.
		
		\subsection{Proposition 3.1 of \cite{Xu1}}
The following statemement is a special case of Proposition 3.1 of \cite{Xu1}.

\begin{quotation}
	Let $(X_A,\Delta_A)\rightarrow (X,Z(f))$ and $(X_B,\Delta_B)\rightarrow (X,Z(f))$ be two dlt modifications and assume that there is a birational map $X_A \dashrightarrow X_B$ that is isomorphic over an open set containing all the generic points of the strata. Then $$Z^{dlt}_{f}((X_A,\Delta_A),s) = Z^{dlt}_{f}((X_B,\Delta_B),s).$$
	\end{quotation}
	
The example in Section \ref{flop} shows that this statement is incorrect; let us indicate more precisely what goes wrong in the proof. 	Let $E_i,i\in I,$ be the irreducible components of $\Delta_A$ that appear with coefficient $1$. Let $W_A$ be a stratum, that is, an irreducible component of $\mathring{E}_J$ for some $J\subset I$. By assumption, there is a unique stratum $W_B$ that corresponds to $W_A$. Define $D_A=\mathrm{Diff}_{\overline{W}_A}(\Delta_A)\vert_{W_A}$ and $D_B=\mathrm{Diff}_{\overline{W}_B}(\Delta_B)\vert_{W_B}$.
	In the proof of Proposition 3.1 in \cite{Xu1}, it is claimed that $(W_A,D_A)$ and $(W_B,D_B)$ are crepant birational. However, this is not true in general. Consider the example from Section \ref{flop}. Set $W_A=\mathring{E}_1\subset X_A$ and $W_B=\mathring{E}_1\subset X_B$. Then, by our computations in Section \ref{flop}, we obtain that $$\mathcal{E}_{st}(W_B,D_B)-\mathcal{E}_{st}(W_A,D_A)=  \L-1 + (\L+1)\frac{\L-1}{\L^{\nicefrac{1}{2}}-1}  \neq 0.$$ This implies that $(W_A,D_A)$ and $(W_B,D_B)$ are not crepant birational, since crepant birational klt pairs have the same stringy motive.

\subsection{Proposition	3.2 of \cite{Xu1}}
The following statemement is a special case of Proposition 3.2 of \cite{Xu1}.
	
\begin{quotation}
	Let $\gamma : (X_2,\Delta_2) \rightarrow (X_1,\Delta_1)$ be a morphism of dlt modifications of $(X,Z(f))$. Let $U \subset X_1$ be an open set containing all log canonical centers of $(X_1,\Delta_1)$ such that $(U,\Delta_1\vert_U)$ is snc. Set $V=\gamma^{-1}(U)$ and assume that $(V,\Delta_2\vert_V ) \rightarrow (U,\Delta_1\vert_U)$ is the blow up of a stratum. Then $$Z^{dlt}_{f}((X_1,\Delta_1), s) = Z^{dlt}_{f}((X_2,\Delta_2), s).$$
	\end{quotation}
	
 Here we reversed the roles of $X_1$ and $X_2$ compared to the phrasing in \cite{Xu1}.	
	 The example in Section \ref{blowup} contradicts this assertion:	
	set	\[
	U=X_1\setminus\Bigl(V(\rho_y,\rho_2)\cup V(\rho_z,\rho_1)\Bigr),\qquad V=\gamma^{-1}(U).
	\]
	Then $\gamma\vert_V\colon V\rightarrow U$ is the blow-up of $U$ with center $E_1\cap E_2\cap U$, but we have shown that $Z^{dlt}_{f}((X_1,\Delta_1), s) \neq Z^{dlt}_{f}((X_2,\Delta_2), s)$.
	In the proof of Proposition 3.2 of \cite{Xu1}, the issue appears to be the reduction to the case where $G$ is a point.
\subsection{Proof of Theorem 1.2 in \cite{Xu1}}
Finally, there appears to be an issue with the proof of Theorem 1.2 in \cite{Xu1}.
In the first step of the argument, one can simply take $\tilde{X}_1=\tilde{X}_2$ to be a common log resolution for $(X_1^{dlt},D_1^{dlt})$ and $(X_2^{dlt},D_2^{dlt})$. The argument does not provide any information on the relation between the dlt zeta functions computed on these two dlt pairs, because, in the notation of \cite{Xu1}, it does not allow us to switch between the values $j=1$ and $j=2$.
	
\section{Further comments}
It is not clear to us if the construction in \cite{Xu1} can be modified to obtain a well-defined invariant of a similar type. Regardless of its independence of the dlt modification, one can still ask whether the dlt motivic zeta function from \cite{Xu1} for a fixed dlt modification can provide insight into the Denef--Loeser motivic zeta function of $f$ and the monodromy conjecture. We have constructed similar examples of non-degenerate polynomials $f$ over $\mathbb{C}$ such that
\begin{itemize}
\item the set of poles of the dlt motivic zeta function depends on the chosen dlt modification;
\item for some choice of a dlt modification, the dlt motivic zeta function has a pole that is not a pole of the Denef--Loeser motivic zeta function;
\item for some choice of a dlt modification, the dlt motivic zeta function has a pole $\alpha\in \mathbb{Q}$ such that $\exp(2\pi i\alpha)$ is not a local monodromy eigenvalue of $f$.
\end{itemize}
This means that, even in the non-degenerate case, the dlt motivic zeta function cannot be used in a straightforward way to find poles of the Denef--Loeser zeta function, and it does not satisfy the direct analogue of the monodromy conjecture.

 More can be said in the two-dimensional case.
		Let $f$ be a non-constant regular function on a connected smooth surface $X$ over $k$. Consider any two dlt modifications $g_i:(X_i,\Delta_i)\rightarrow (X,Z(f)_{red}),i\in\{1,2\}$. It can be verified by elementary techniques that
		\[
		Z_{f}^{dlt}((X_1,\Delta_1),s)=Z_{f}^{dlt}((X_2,\Delta_2),s).
		\]
		Thus, in dimension two, the dlt motivic zeta function is well-defined. Furthermore, one can check that the dlt motivic zeta function has exactly the same set of poles as the Denef--Loeser motivic zeta function. Details will appear in the second-named author's PhD thesis.


	\bibliographystyle{alpha}
	\bibliography{references1}
\end{document}